\documentclass[reqno,oneside]{amsart}

\usepackage[danish,english]{babel} 
\usepackage[latin1]{inputenc} 
\usepackage[T1]{fontenc} 
\usepackage{lmodern}
\usepackage{amsmath,amsthm,amssymb,amsfonts,mathrsfs,latexsym} 
\usepackage{graphicx} 
\usepackage{verbatim} 
\usepackage[all]{xy} 
\usepackage[hyperfootnotes=false, plainpages=false, pdfpagelabels, colorlinks, linkcolor=red, citecolor=blue, urlcolor=blue, hypertexnames=true]{hyperref} 
\usepackage{multirow}

\makeatletter
\newtheorem*{rep@theorem}{\rep@title}
\newcommand{\newreptheorem}[2]{%
\newenvironment{rep#1}[1]{%
 \def\rep@title{#2 \ref{##1}}%
 \begin{rep@theorem}}%
 {\end{rep@theorem}}}
\makeatother

\newtheorem{thm}{Theorem}

\newreptheorem{thm}{Theorem}
\newtheorem{lem}[thm]{Lemma}

\newtheorem{cor}[thm]{Corollary}
\newtheorem*{thm*}{Theorem}
\newtheorem*{problem*}{Problem}

\newtheorem*{claim*}{Claim}
\theoremstyle{definition}

\newcommand{\mc}{\mathcal}

\newcommand{\mr}{\mathrm}

\newcommand{\C}{\mathbb{C}}
\newcommand{\N}{\mathbb{N}}

\newcommand{\R}{\mathbb{R}}

\newcommand{\la}{\langle}
\newcommand{\ra}{\rangle}
\renewcommand{\epsilon}{\varepsilon}
\renewcommand{\phi}{\varphi}

\renewcommand{\hat}{\widehat}

\renewcommand{\bar}{\overline}

\newcommand{\Tube}{\mr{Tube}}

\DeclareFontFamily{U}{mathx}{\hyphenchar\font45}
\DeclareFontShape{U}{mathx}{m}{n}{
      <5> <6> <7> <8> <9> <10>
      <10.95> <12> <14.4> <17.28> <20.74> <24.88>
      mathx10
      }{}
\DeclareSymbolFont{mathx}{U}{mathx}{m}{n}
\DeclareFontSubstitution{U}{mathx}{m}{n}
\DeclareMathAccent{\widecheck}{0}{mathx}{"71}
\DeclareMathAccent{\wideparen}{0}{mathx}{"75}

\numberwithin{equation}{section}

\begin{document}
\selectlanguage{english} 

\begin{abstract}
We give a contractive Schur multiplier characterization of locally compact groups coarsely embeddable into Hilbert spaces. Consequently, all locally compact groups whose weak Haagerup constant is 1 embed coarsely into Hilbert spaces, and hence the Baum-Connes assembly map with coefficients is split-injective for such groups.
\end{abstract}


\title[A Schur multiplier characterization of coarse embeddability]{\texorpdfstring{A Schur multiplier characterization\\of coarse embeddability}{A Schur multiplier characterization of coarse embeddability}}

\author{S{\o}ren Knudby}
\address{Department of Mathematical Sciences, University of Copenhagen,
\newline Universitetsparken 5, DK-2100 Copenhagen \O, Denmark}
\email{knudby@math.ku.dk}

\author{Kang Li}
\address{Department of Mathematical Sciences, University of Copenhagen,
\newline Universitetsparken 5, DK-2100 Copenhagen \O, Denmark}
\email{kang.li@math.ku.dk}

\thanks{Both authors are supported by ERC Advanced Grant no.~OAFPG 247321 and the Danish National Research Foundation through the Centre for Symmetry and Deformation (DNRF92).}
\thanks{\textit{Mathematics Subject Classification (2010): }43A22, 43A35, 46L80}
\thanks{\textit{Keywords: }Coarse embedding, Schur multipliers, Baum-Connes conjecture}

\date{\today}
\maketitle
\parindent 0cm
\parskip 4pt


In this note we study coarse embeddability of locally compact groups into Hilbert spaces. An important application of this concept in \cite{MR1728880}, \cite{MR1905840} and \cite{deprez-li-A-UE} is that the Baum-Connes assembly map with coefficients is split-injective for all locally compact groups that embed coarsely into a Hilbert space (see \cite{MR1292018} and \cite{MR1907596} for more information about the Baum-Connes assembly map). Here, we give a contractive Schur multiplier characterization of locally compact groups coarsely embeddable into Hilbert spaces (see also \cite[Theorem~5.3]{MR2945214} for the discrete case), and this characterization can be regarded as an answer to the non-equivariant version of \cite[Question~1.5]{KL-simple}. As a result, any locally compact group with weak Haagerup constant 1 embeds coarsely into a Hilbert space and hence the Baum-Connes assembly map with coefficients is split-injective for all these groups.

Let $G$ be a $\sigma$-compact, locally compact group. A \emph{(left) tube} in $G\times G$ is a subset of $G\times G$ contained in a set of the form
$$
\Tube(K) = \{(x,y)\in G\times G \mid x^{-1}y\in K\}
$$
where $K$ is any compact subset of $G$.
Following \cite[Definition~3.6]{MR1926869}, we say that a map $u$ from $G$ into a Hilbert space $H$ is a \emph{coarse embedding} if $u$ satisfies the following two conditions:

\begin{itemize}
	\item for every compact subset $K$ of $G$ there exists $R>0$ such that
\begin{align*}
(s,t)\in \Tube(K)\implies \|u(s)-u(t)\|\leq R;
\end{align*}
	\item for every $R>0$ there exists a compact subset $K$ of $G$ such that
\begin{align*}
\|u(s)-u(t)\|\leq R \implies (s,t)\in \Tube(K).
\end{align*}

\end{itemize}
We say that a group $G$ \emph{embeds coarsely into a Hilbert space} or \emph{admits a coarse embedding into a Hilbert space} if there exist a Hilbert space $H$ and a coarse embedding $u:G\rightarrow H$. Note that a coarse embedding need not be injective, and we also do not require it to be continuous.

Every second countable, locally compact group $G$ admits a proper left-invariant metric $d$, which is unique up to coarse equivalence (see \cite{MR0348037} and \cite{HP}). So the preceding definition is equivalent to Gromov's notion of coarse embeddability of the metric space $(G,d)$ into Hilbert spaces. We refer to \cite[Section 3]{deprez-li-A-UE} for more on coarse embeddability into Hilbert spaces for locally compact groups.

A kernel $\phi\colon G\times G\to\C$ is a \emph{Schur multiplier} if for every bounded operator $A = (a_{x,y})_{x,y\in G}\in B(\ell^2(G))$, the matrix $(\phi(x,y)a_{x,y})_{x,y\in G}$ again defines a bounded operator, denoted $M_\phi A$, on $\ell^2(G)$. In this case, it follows from the closed graph theorem that $M_\phi$ in fact defines a \emph{bounded} operator $B(\ell^2(G))\to B(\ell^2(G))$, and the \emph{Schur norm} $\|\phi\|_S$ of $\phi$ is defined to be the operator norm of $M_\phi$.

A kernel $\phi\colon G\times G\to\C$ \emph{tends to zero off tubes}, if for any $\epsilon > 0$ there is a tube $T\subseteq G\times G$ such that $|\phi(x,y)| < \epsilon$ whenever $(x,y)\notin T$.
Note that if $\phi\colon G\to\C$ is a function, then $\phi$ vanishes at infinity (written $\phi\in C_0(G)$), if and only if the associated kernel $\hat\phi\colon G\times G\to\C$ defined by $\hat\phi(x,y) = \phi(x^{-1}y)$ tends to zero off tubes.

\begin{thm}\label{thm:Schur-CE}
Let $G$ be a $\sigma$-compact, locally compact group. The following are equivalent.
\begin{enumerate}
	\item $G$ embeds coarsely into a Hilbert space.
	\item There exists a sequence of Schur multipliers $\phi_n\colon G\times G\to\C$ such that
	\begin{itemize}
		\item $\|\phi_n\|_S\leq 1$ for every natural number $n$;
		\item each $\phi_n$ tends to zero off tubes;
		\item $\phi_n\to 1$ uniformly on tubes.
	\end{itemize}
\end{enumerate}
If any of these conditions holds, one can moreover arrange that the coarse embedding is continuous and that each $\phi_n$ is continuous.
\end{thm}


It is well-known that the notion of coarse embeddability into Hilbert spaces can be characterized by positive definite kernels (see \cite[Theorem~2.3]{MR1876896} for the discrete case and \cite[Theorem~1.5]{deprez-li-A-UE2} for the locally compact case).

Following \cite{K-WH}, $G$ has the \emph{weak Haagerup property with constant 1}, if there is a sequence of continuous functions $\phi_n\in C_0(G)$ convering uniformly to $1$ on compact subsets of $G$ and such that the associated kernels $\hat\phi_n\colon G\times G\to\C$ are Schur multipliers with $\|\hat\phi_n\|_S\leq 1$.

From Theorem~\ref{thm:Schur-CE} together with \cite[Theorem~3.5]{deprez-li-A-UE} we immediately obtain the following. 
\begin{cor}
If $G$ is a $\sigma$-compact, locally compact group with the weak Haagerup property with constant $1$, then $G$ embeds coarsely into a Hilbert space. If $G$ is moreover second countable, then in particular the Baum-Connes assembly map with coefficients is split-injective.
\end{cor}


We now turn to the proof of Theorem~\ref{thm:Schur-CE}. It is not hard to see that the countability assumption in \cite[Proposition~4.3]{MR3146826} is superfluous. We thus record the following (slightly more general) version of \cite[Proposition~4.3]{MR3146826}. 
\begin{lem}\label{lem:split}
Let $G$ be a group with a symmetric kernel $k\colon G\times G\to[0,\infty)$. The following are equivalent.
\begin{enumerate}
	\item For every $t>0$ one has $\|e^{-tk}\|_S \leq 1$.
	\item There exist a real Hilbert space $\mc H$ and maps $R,S\colon G\to\mc H$ such that
$$
k(x,y) = \|R(x)-R(y)\|^2 + \|S(x)+S(y)\|^2 \quad\text{ for every } x,y\in G.
$$
\end{enumerate}
\end{lem}

Recall that a kernel $k\colon G\times G\to\R$ is \emph{conditionally negative definite} if $k$ is symmetric ($k(x,y) = k(y,x)$), vanishes on the diagonal ($k(x,x) = 0$) and
$$
\sum_{i,j=1}^n c_i c_j k(x_i,x_j) \leq 0
$$
for any finite sequences $x_1,\ldots,x_n\in G$ and $c_1,\ldots,c_n\in\R$ such that $\sum_{i=1}^n c_i = 0$. It is well-known that $k$ is conditionally negative definite if and only if there is a function $u$ from $G$ to a real Hilbert space such that $k(x,y) = \|u(x) - u(y)\|^2$.

A kernel $k\colon G\times G\to\C$ is called \emph{proper}, if the set $\{(x,y)\in G\times G\mid |k(x,y)|\leq R\}$ is a tube for every $R>0$.

Theorem~\ref{thm:Schur-CE} is contained in Theorem~\ref{thm:CE-equivalent} below, which extends both \cite[Theorem~5.3]{MR2945214} and \cite[Theorem~1.5]{deprez-li-A-UE2} in different directions. An important ingredient in the proof of Theorem~\ref{thm:CE-equivalent} is the following result (which generalizes without change from the second countable case to the $\sigma$-compact case).

\begin{thm*}[{\cite[Theorem~3.4]{deprez-li-A-UE}}]
Let $G$ be a $\sigma$-compact, locally compact group. The following are equivalent.
\begin{enumerate}
\item The group $G$ embeds coarsely into a Hilbert space.
\item There is a continuous conditionally negative definite kernel $h\colon G\times G\to \R$ which is proper and bounded on tubes.
\end{enumerate}
\end{thm*}

\begin{thm}\label{thm:CE-equivalent}
Let $G$ be a $\sigma$-compact, locally compact group. The following are equivalent.
\begin{enumerate}
	\item The group $G$ embeds coarsely into a Hilbert space.
	\item There exists a sequence of (not necessarily continuous) Schur multipliers $\phi_n\colon G\times G\to\C$ such that
	\begin{itemize}
		\item $\|\phi_n\|_S\leq 1$ for every natural number $n$;
		\item each $\phi_n$ tends to zero off tubes;
		\item $\phi_n\to 1$ uniformly on tubes.
	\end{itemize}
	\item There exists a (not necessarily continuous) symmetric kernel $k\colon G\times G\to[0,\infty)$ which is proper, bounded on tubes and satisfies $\|e^{-tk}\|_S\leq 1$ for all $t> 0$.
	\item There exists a (not necessarily continuous) conditionally negative definite kernel $h\colon G\times G\to \R$ which is proper and bounded on tubes.
\end{enumerate}
Moreover, if any of these conditions holds, one can arrange that the coarse embedding in \emph{(1)}, each Schur multiplier $\phi_n$ in \emph{(2)}, the symmetric kernel $k$ in \emph{(3)} and the conditionally negative definite kernel $h$ in \emph{(4)} are continuous.
\end{thm}
\begin{proof}
We show (1)$\iff$(4)$\iff$(3)$\iff$(2).

That (1) implies (4) with $h$ continuous follows directly from \cite[Theorem~3.4]{deprez-li-A-UE}.

Suppose (4) holds. By the GNS construction
there are a real Hilbert space $\mc H$ and a map $u\colon G\to\mc H$ such that
$$
h(x,y) = \|u(x)-u(y)\|^2.
$$
It is easy to check that the assumptions on $h$ imply that $u$ is a coarse embedding. Thus (1) holds.

That (4) implies (3) follows with $k = h$ using Schoenberg's Theorem and the fact that normalized positive definite kernels are Schur multipliers of norm $1$. Note also that conditionally negative definite kernels are symmetric and take only non-negative values.

Suppose (3) holds. We show that (4) holds. From Lemma~\ref{lem:split} we see that there are a real Hilbert space $\mc H$ and maps $R,S\colon G\to\mc H$ such that
$$
k(x,y) = \|R(x)-R(y)\|^2 + \|S(x)+S(y)\|^2 \quad\text{ for every } x,y\in G.
$$
As $k$ is bounded on tubes, the map $S$ is bounded. If we let
$$
h(x,y) = \|R(x)-R(y)\|^2,
$$
then it is easily checked that $h$ is proper and bounded on tubes, since $k$ has these properties and $S$ is bounded. It is also clear that $h$ is conditionally negative definite. Thus (4) holds.

If (3) holds, we set $\phi_n = e^{-k/n}$ when $n\in\N$. It is easy to check that the sequence $\phi_n$ has the desired properties so that (2) holds.

Finally, suppose (2) holds. We verify (3). Essentially, we use the same standard argument as in the proof of \cite[Proposition~4.4]{K-WH} and \cite[Theorem~2.1.1]{MR1852148}.

Since $G$ is locally compact and $\sigma$-compact, it is the union of an increasing sequence $(U_n)_{n=1}^\infty$ of open sets such that the closure $K_n$ of $U_n$ is compact and contained in $U_{n+1}$ (see \cite[Proposition~4.39]{MR1681462}). Fix an increasing, unbounded sequence $(\alpha_n)$ of positive real numbers and a decreasing sequence $(\epsilon_n)$ tending to zero such that $\sum_n \alpha_n\epsilon_n$ converges. By assumption, for every $n$ we can find a Schur multiplier $\phi_n$ tending to zero off tubes and such that $\|\phi_n\|_S \leq 1$ and
$$
\sup_{(x,y)\in \Tube(K_n)} |\phi_n(x,y) - 1| \leq \epsilon_n/2.
$$
Upon replacing $\phi_n$ by $|\phi_n|^2$ one can arrange that $0 \leq \phi_n \leq 1$ and
$$
\sup_{(x,y)\in \Tube(K_n)} |\phi_n(x,y) - 1| \leq \epsilon_n.
$$
Define kernels $\psi_i:G\times G\to[0,\infty[$ and $\psi:G\times G\to[0,\infty[$ by
$$
\psi_i(x,y) = \sum_{n=1}^i \alpha_n (1-\phi_n(x,y)), \qquad \psi(x,y) = \sum_{n=1}^\infty \alpha_n (1-\phi_n(x,y)).
$$
It is easy to see that $\psi$ is well-defined, bounded on tubes and $\psi_i\to\psi$ pointwise (even uniformly on tubes, but we do not need that).

To see that $\psi$ is proper, let $R > 0$ be given. Choose $n$ large enough such that $\alpha_n \geq 2R$. As $\phi_n$ tends to zero off tubes, there is a compact set $K \subseteq G$ such that $|\phi_n(x,y)| < 1/2$ whenever $(x,y)\notin \Tube(K)$. Now if $\psi(x,y) \leq R$, then $\psi(x,y) \leq \alpha_n / 2$, and in particular $\alpha_n (1-\phi_n(x,y)) \leq \alpha_n/2$, which implies that $1-\phi_n(x,y) \leq 1/2$. We have thus shown that
$$
\{(x,y)\in G\times G \mid \psi(x,y) \leq R \} \subseteq \{(x,y)\in G\times G \mid 1-\phi_n(x,y) \leq 1/2 \}\subseteq \Tube(K),
$$
and $\psi$ is proper.

We now show that $\|e^{-t\psi}\|_S \leq 1$ for every $t>0$. 
Since $\psi_i$ converges pointwise to $\psi$, it will suffice to prove that $\|e^{-t\psi_i}\|_S \leq 1$, because the set of Schur multipliers of norm at most $1$ is closed under pointwise limits. Since
$$
e^{-t\psi_i} = \prod_{n=1}^i e^{-t \alpha_n (1-\phi_n) },
$$
it is enough to show that $e^{-t \alpha_n (1-\phi_n) }$ has Schur norm at most $1$ for each $n$. And this is clear:
$$
\| e^{-t \alpha_n (1-\phi_n) } \|_S = e^{-t \alpha_n} \| e^{t \alpha_n \phi_n } \|_S \leq e^{-t \alpha_n}  e^{t \alpha_n \| \phi_n\|_S } \leq 1.
$$

The only thing missing is that $\psi$ need not be symmetric. Put $k = \psi + \widecheck\psi$ where $\widecheck\psi(x,y) = \psi(y,x)$. Clearly, $k$ is symmetric, bounded on tubes and proper. Finally, for every $t>0$
$$
\|e^{-tk}\|_S \leq \|e^{-t\psi}\|_S \|e^{-t\widecheck\psi}\|_S \leq 1,
$$
since $\|\widecheck\phi\|_S = \|\phi\|_S$ for every Schur multiplier $\phi$.

Finally, the statements about continuity follow from \cite[Theorem~3.4]{deprez-li-A-UE} and the explicit constructions used in our proof of (1)$\implies$(4)$\implies$(3)$\implies$(2).
\end{proof}



\begin{thebibliography}{10}

\bibitem{MR1926869}
Claire Anantharaman-Delaroche.
\newblock Amenability and exactness for dynamical systems and their
  {$C^\ast$}-algebras.
\newblock {\em Trans. Amer. Math. Soc.}, 354(10):4153--4178 (electronic), 2002.

\bibitem{MR1292018}
Paul Baum, Alain Connes, and Nigel Higson.
\newblock Classifying space for proper actions and {$K$}-theory of group
  {$C^\ast$}-algebras.
\newblock In {\em {$C^\ast$}-algebras: 1943--1993 ({S}an {A}ntonio, {TX},
  1993)}, volume 167 of {\em Contemp. Math.}, pages 240--291. Amer. Math. Soc.,
  Providence, RI, 1994.

\bibitem{MR1852148}
Pierre-Alain Cherix, Michael Cowling, Paul Jolissaint, Pierre Julg, and Alain
  Valette.
\newblock {\em Groups with the {H}aagerup property}, volume 197 of {\em
  Progress in Mathematics}.
\newblock Birkh\"auser Verlag, Basel, 2001.
\newblock Gromov's a-T-menability.

\bibitem{deprez-li-A-UE2}
Steven Deprez and Kang Li.
\newblock Permanence properties of property {A} and coarse embeddability for
  locally compact groups.
\newblock Preprint, arXiv:1403.7111, 2014.

\bibitem{deprez-li-A-UE}
Steven Deprez and Kang Li.
\newblock Property {A} and uniform embedding for locally compact groups.
\newblock Preprint, to appear in J. Noncommut. Geom., 2015.

\bibitem{MR2945214}
Zhe Dong and Zhong-Jin Ruan.
\newblock A {H}ilbert module approach to the {H}aagerup property.
\newblock {\em Integral Equations Operator Theory}, 73(3):431--454, 2012.

\bibitem{MR1681462}
Gerald~B. Folland.
\newblock {\em Real analysis}.
\newblock Pure and Applied Mathematics (New York). John Wiley \& Sons Inc., New
  York, second edition, 1999.
\newblock Modern techniques and their applications, A Wiley-Interscience
  Publication.

\bibitem{MR1876896}
Erik Guentner and Jerome Kaminker.
\newblock Exactness and the {N}ovikov conjecture.
\newblock {\em Topology}, 41(2):411--418, 2002.

\bibitem{HP}
Uffe Haagerup and Agata Przybyszewska.
\newblock Proper metrics on locally compact groups, and proper affine isometric
  actions on {B}anach spaces.
\newblock Preprint, arXiv:math/0606794, 2006.

\bibitem{MR3146826}
S{\o}ren Knudby.
\newblock Semigroups of {H}erz-{S}chur multipliers.
\newblock {\em J. Funct. Anal.}, 266(3):1565--1610, 2014.

\bibitem{K-WH}
S{\o}ren Knudby.
\newblock The weak {H}aagerup property.
\newblock Preprint, to appear in Trans. Amer. Math. Soc., 2015.

\bibitem{KL-simple}
S{\o}ren Knudby and Kang Li.
\newblock Approximation properties of simple {L}ie groups made discrete.
\newblock Preprint, to appear in J. Lie Theory, 2015.

\bibitem{MR1905840}
G.~Skandalis, J.~L. Tu, and G.~Yu.
\newblock The coarse {B}aum-{C}onnes conjecture and groupoids.
\newblock {\em Topology}, 41(4):807--834, 2002.

\bibitem{MR0348037}
Raimond~A. Struble.
\newblock Metrics in locally compact groups.
\newblock {\em Compositio Math.}, 28:217--222, 1974.

\bibitem{MR1907596}
Alain Valette.
\newblock {\em Introduction to the {B}aum-{C}onnes conjecture}.
\newblock Lectures in Mathematics ETH Z\"urich. Birkh\"auser Verlag, Basel,
  2002.
\newblock From notes taken by Indira Chatterji, With an appendix by Guido
  Mislin.

\bibitem{MR1728880}
Guoliang Yu.
\newblock The coarse {B}aum-{C}onnes conjecture for spaces which admit a
  uniform embedding into {H}ilbert space.
\newblock {\em Invent. Math.}, 139(1):201--240, 2000.

\end{thebibliography}

\end{document}